\documentclass[10 pt,a4paper,twoside,reqno]{amsart}
\usepackage{amssymb, amsmath, amsthm, amsfonts, amscd, enumerate, verbatim, calc}
\usepackage{datetime}

\newtheorem{theorem}{Theorem}[section]
\newtheorem{prop}[theorem]{Proposition}
\newtheorem{lemma}[theorem]{Lemma}
\newtheorem{cor}[theorem]{Corollary}
\theoremstyle{definition}
\newtheorem{defn}[theorem]{Definition}

\newtheorem{note}[theorem]{Note}

\theoremstyle{remark}
\newtheorem{rem}[theorem]{Remark}

\numberwithin{equation}{section}

\title{On the denseness of minimum attaining operators}
\author{S. H. Kulkarni}
\address{Department of Mathematics, Indian Institute of Technology Madras, Taramani, Chennai, Tamilnadu  600 036}
\email{shk@iitm.ac.in}

\author{G. Ramesh}
\address{Department of Mathematics, Indian Institute of Technology - Hyderabad, Kandi, Sangareddy (M), Medak (Dist), Telangana, India 502 285}
\email{rameshg@iith.ac.in}
\date{\currenttime \;  \today}
\begin{document}
\maketitle

\begin{abstract}
Let $H_1,H_2$ be complex Hilbert spaces and $T$ be a densely defined closed linear operator (not necessarily bounded). It is proved that for each $\epsilon>0$, there exists a bounded operator $S$ with $\|S\|\leq \epsilon$ such that $T+S$
is minimum attaining. Further, if $T$ is bounded below, then $S$ can be chosen to be rank one.
\end{abstract}

\section{introduction}
It is well known that the set of all norm attaining operators defined between two complex Hilbert spaces is norm dense in the space of all bounded linear operators
defined between complex Hilbert spaces. This result is even true for operators defined between Banach spaces, when the domain space is reflexive, which is proved by Lindenstrauss \cite{lindenstrauss}. A simple proof of this fact, in the case of Hilbert space operators is given by Enflo et al. in \cite{enfloetal}. Moreover, the authors proved that rank one perturbation of a bounded operator can be made as norm attaining operator.

Similar to the norm attaining operators, bounded operators that attain their minimum modulus is introduced in \cite{carvajalneves2}. The unbounded case is dealt in \cite{shkgr7} and the authors established basic properties of minimum attaining closed densely defined operators.

It is very natural to ask whether the Lindenstrauss theorem is true in case of minimum attaining operators. In this article we answer this question affirmatively. We show that the set of all minimum attaining densely defined closed operators is dense in the class of densely defined closed operators with respect to the gap metric. As a consequence, we can conclude the same is true for bounded operators with respect to the operator norm. In a special case, we also show that rank one perturbations lead to minimum attaining operators. This leads to the perturbations of minimum attaining operators.

In the second section we recall some basic definitions and results which we need for proving our main results. In the third section we prove the Lindenstrauss type theorem for minimum attaining operators.

\section{Notations and Preliminaries}
 Throughout we consider infinite dimensional complex Hilbert spaces which will be denoted by $H, H_1,H_2$ etc. The inner product and the
induced norm are denoted  by  $\langle \cdot\rangle$ and
$||.||$, respectively. The closure of a subspace $M$ of $H$ is denoted by $\overline{M}$. We denote the unit sphere of $M$ by $S_M={\{x\in M:\|x\|=1}\}$.

If $M$ is a closed subspace of a Hilbert space $H$, then $P_M$ denotes the orthogonal projection $P_M:H\rightarrow H$ with range $M$.

Let $T$ be a linear operator with domain $D(T)$ (a subspace of $H_1$) and taking values in $H_2$. If $D(T)$ is dense in $H_1$, then $T$ is called a densely defined operator. The  graph $G(T)$ of $T$ is defined by  $G(T):={\{(Tx,x):x\in D(T)}\}\subseteq H_1\times H_2$. If $G(T)$ is
closed, then $T$ is called a closed operator.  Equivalently, $T$ is closed if $(x_n)$ is a sequence in  $D(T)$  such that  $x_n\rightarrow x\in H_1$ and $Tx_n\rightarrow y\in H_2$, then $x\in D(T)$ and $Tx=y$.

For a densely defined linear operator $T$, there exists a unique linear operator (in fact, a closed operator) $T^*:D(T^*)\rightarrow H_1$, with
\begin{equation*}
D(T^*):={\{y\in H_2: x\rightarrow \langle Tx,y\rangle \, \text{for all}\, x\in D(T)\,\text{is continuous}}\}\subseteq H_2
\end{equation*}
 satisfying
 \begin{equation*}
 \langle Tx,y\rangle =\langle x,T^*y\rangle \text{ for all}\; x\in D(T) \text{ and}\; y\in D(T^*).
 \end{equation*}

 We say $T$ to be bounded if there exists $M>0$ such that $\|Tx\|\leq M\|x\|$ for all $x\in D(T)$. Note that if $T$ is densely defined and bounded then $T$ can be extended to all of $H_1$ in a unique way.

By the closed graph Theorem \cite{rud}, an everywhere defined
closed operator is bounded.  Hence the domain of an unbounded closed operator is a proper subspace of a Hilbert space.

The space of all bounded linear operators between $H_1$ and $H_2$ is
denoted by $\mathcal B(H_1,H_2)$ and the class of all closed linear operators between $H_1$ and $H_2$ is denoted by $\mathcal C(H_1,H_2)$. We write $\mathcal B(H,H)=\mathcal B(H)$  and $\mathcal C(H,H)=\mathcal C(H)$.

If $T\in \mathcal B(H_1,H_2)$ is such that for every bounded sequence $(x_n)$ of $H_1$,  $(Tx_n)$ has a convergent subsequence in $H_2$, then $T$ is called a compact operator. Equivalently, $T$ is compact if and only if for every bounded set $B$ of $H_1$, $T(B)$ is pre compact in $H_2$.

If $T\in \mathcal C(H_1,H_2)$, then  the null space and the range space of $T$ are denoted by $N(T)$ and $R(T)$, respectively and the space $C(T):=D(T)\cap N(T)^\bot$ is called the carrier of $T$.
In fact, $D(T)=N(T)\oplus^\bot C(T)$ \cite[page 340]{ben}. Here $\oplus^{\bot}$ denote the orthogonal  direct sum of subspaces.

A subspace $D$ of $D(T)$ is called a core for $T$ if for any $x\in D(T)$, there exists a sequence $(x_n)\subset D$ such that $\displaystyle \lim_{n\rightarrow \infty}x_n=x$ and $\displaystyle \lim_{n\rightarrow \infty}Tx_n=Tx$. In other words, $D$ is dense in the graph norm, which is defined by $\||x\||:=\|x\|+\|Tx\|$ for all $x\in D(T)$. It is a well known fact that $D(T^*T)$ is a core for $T$ (see \cite[Definition 1.5, page 7, Proposition 3.18, page 47]{schmudgen} for details).


 Let $S,T\in \mathcal C(H)$ be densely defined operators with domains $D(S)$ and $D(T)$, respectively. Then $S+T$ is an operator with domain $D(S+T)=D(S)\cap D(T)$ defined by $(S+T)(x)=Sx+Tx$ for all $x\in D(S+T)$.
 The operator $ST$  has the domain $D(ST)={\{x\in D(T): Tx\in D(S)}\}$ and is defined as $(ST)(x)=S(Tx)$ for all $x\in D(ST)$.

If $S$ and $T$ are closed operators  with the property that $D(T)\subseteq D(S)$ and $Tx=Sx$ for all $x\in D(T)$, then $S$ is called the restriction of $T$ and $T$ is called an extension of $S$.

A densely defined operator $T\in \mathcal C(H)$ is said to be
\begin{enumerate}
 \item normal if $T^*T=TT^*$
  \item self-adjoint if $T=T^*$
 \item positive if  $T=T^*$ and $\langle Tx,x\rangle \geq 0$ for all $x\in D(T)$.
\end{enumerate}

Let $V\in \mathcal B(H_1,H_2)$. Then $V$ is called
\begin{enumerate}
\item an isometry if $\|Vx\|=\|x\|$ for all $x\in H_1$
\item a partial isometry if $V|_{N(V)^{\bot}}$ is an isometry. The space $N(V)^{\bot}$ is called the initial space or the initial domain and the space $R(V)$ is called the final space or the final domain of $V$.
\end{enumerate}

\begin{defn}\label{normattainingdefn}
Let $T\in \mathcal B(H_1,H_2)$. Then $T$ is said to be {\it norm attaining} if there exists $x_0\in S_{N(T)^{\bot}}\subseteq H_1$ such that $\|Tx_0\|=\|T\|$.
\end{defn}

\begin{defn}\cite{ben, goldberg}\label{minmmodulus}
 Let $T\in \mathcal C(H_1,H_2)$ be densely defined. Then
 \begin{equation*}
 m(T)=\inf{\{\|Tx\|: x\in S_{D(T)}}\}
  \end{equation*}
 is  called the {\it minimum modulus} of $T$.
\end{defn}

\begin{defn}\cite[Definition 3.1]{shkgr7}\label{minmattainingdefn}
Let $T\in \mathcal C(H_1,H_2)$ be densely defined. Then we say $T$ to be  {\it minimum attaining} if there exists $x_0\in S_{D(T)}$ such that $\|Tx_0\|=m(T)$.

\end{defn}
We denote the class of minimum attaining closed operators between $H_1$ and $H_2$ by $\mathcal M_c(H_1,H_2)$ and $\mathcal M_c(H,H)$ by $\mathcal M_c(H)$.

\begin{note}
If $T\in \mathcal C(H_1,H_2)$ is densely defined and $N(T)\neq {\{0}\}$, then $m(T)=0$ and there exists $x\in S_{N(T)}$ such that $Tx=0$. Hence $T\in \mathcal M_c(H_1,H_2)$.
\end{note}
\begin{rem}
If $T\in \mathcal C(H_1,H_2)$ is densely defined, then  $m(T)>0$ if and only if $R(T)$ is closed and $T$ is one-to-one.
\end{rem}

\begin{theorem}\label{squareroot}\cite[theorem 13.31, page 349]{rud}\cite[Theorem 4, page 144]{birmannsolomyak}
Let $T\in \mathcal C(H)$ be densely defined and positive. Then there exists a unique positive operator $S$ such that
$T=S^2$. The operator $S$ is called the square root of $T$ and is denoted by $S=T^\frac{1}{2}$.
\end{theorem}
For $T\in \mathcal C(H_1,H_2)$, the operator  $|T|:=(T^*T)^\frac{1}{2}$ is called the modulus of $T$. Moreover,  $D(|T|)=D(T),\; N(|T|)=N(T)$ and $\overline{R(|T|)}=\overline{R(T^*)}$.
 As $\|Tx\|=\||T|x\|$ for all $x\in D(T)$,  we can conclude that  $m(T)=m(|T|)$ and $T\in \mathcal M_c(H_1,H_2)$ if and only if $|T|\in \mathcal M_c(H_1)$.

\begin{theorem}\cite[Theorem 2, page 184]{birmannsolomyak}
Let $T\in \mathcal C(H_1,H_2)$ be densely defined. Then there exists a unique partial isometry $V:H_1\rightarrow H_2$ with the initial space $\overline{R(T^*)}$ and the final space $\overline{R(T)}$ such that $T=V|T|$.
\end{theorem}
\begin{defn}\cite[page 346]{rud}
 Let $T\in \mathcal C(H)$ be densely defined. The resolvent of $T$ is defined by
 \begin{equation*}
 \rho(T):={\{\lambda \in \mathbb C: T-\lambda I:D(T)\rightarrow H\; \text{is invertible and}\; (T-\lambda I)^{-1}\in \mathcal B(H)}\}
 \end{equation*}
and
\begin{align*}
\sigma(T):&=\mathbb C \setminus \rho(T)\\
\sigma_p(T):&={\{\lambda \in \mathbb C: T-\lambda I:D(T)\rightarrow H \; \text{is not one-to-one}}\},
\end{align*}
are called the  spectrum and the point spectrum of $T$, respectively.
\end{defn}

\begin{defn}\cite[Definition 8.3 page 178]{schmudgen}
Let $T=T^* \in \mathcal{C}(H)$. Then the \textit{discrete spectrum} $\sigma_d(T)$ of $T$ is defined as the set of all eigenvalues of $T$ with finite multiplicities which are isolated points of the spectrum $\sigma(T)$ of $T$. The complement set $\sigma_{ess}(T):=\sigma(T)\setminus \sigma_d(T)$ is called the \textit{essential spectrum} of $T$.
\end{defn}

The essential spectrum is stable under compact perturbations.

\begin{prop}[Weyl's theorem]\label{weyltheorem}\cite[Corollary 8.16, page 182]{schmudgen}
Let $A\in \mathcal C(H)$ be self-adjoint and $C$ is compact, self-adjoint. Then $\sigma_{ess}(A+C)=\sigma_{ess}(A)$.
\end{prop}

\begin{lemma}\cite{schock,gr92,ped}\label{boundedclosure}
 Let $T\in \mathcal C(H_1,H_2)$ be densely defined. Denote $\check T=(I+T^*T)^{-1}$ and $\widehat T=(I+TT^*)^{-1}$. Then
\begin{enumerate}
\item $\check T\in \mathcal B(H_1)$, $\widehat T\in \mathcal
B(H_2)$
\item $\widehat TT\subseteq T\check T$,\quad $||T\check T||\leq \displaystyle
\frac{1}{2}$ and $\check TT^*\subseteq T^*\widehat T$,\quad
$||T^*\widehat T||\leq \displaystyle \frac{1}{2}$
\item \label{contfnlcalculus-groetsch} if $g:[0,1]\rightarrow \mathbb C$ is a continuous function, then
\begin{itemize}
\item[(a)] $T^*g(\Hat {T})y=g(\check{T})T^*y$ for all $y\in D(T^*)$
\item[(b)] $ Tg(\check{T})x=g(\Hat{T})Tx$ for all $x\in D(T)$.
\end{itemize}
\end{enumerate}
\end{lemma}

\subsection{Gap Metric}

 The gap between closed subspaces $M$ and $N$ of $H$ is defined by $\theta(M,N):=\|P_M-P_N\|$. This defines a metric on the class of closed subspaces of $H$,
known as the gap metric. The topology induced by the gap metric is known as the gap topology. We have the following alternative formula for the gap;

 \begin{equation*}
 \theta(M,N):=\max
\Big\{||P_M(I-P_N)||,||P_N(I-P_M)||\Big\}.
\end{equation*}
For the details we refer to \cite[Page 70]{akh}.

Let   $A,B\in \mathcal C(H_1,H_2)$ be densely defined.  Then $G(A)$ and $G(B)$ are closed subspaces of $H_1\times H_2$.
The gap between $A$ and $B$ is defined as

\begin{equation*}
\theta(A,B)=\|P_{G(A)}-P_{G(B)}\|,
\end{equation*}
where $P_{M}:H_1\times H_2 \rightarrow H_1\times H_2$, is an orthogonal projection onto the closed subspace $M$ of $H_1\times H_2$.  This defines a  metric on the class of closed operators and induced
 topology is known as the gap topology. The gap topology restricted to the space of bounded linear operators coincides with the norm topology.
Also the convergence with respect to the gap metric on the set of self-adjoint bounded operators coincide with  the resolvent convergence \cite[Chapter VII, page 235]{reedsimon}).

We have the following formula for the gap between two closed operators;

\begin{theorem}\cite{shkgrgap}\label{shkgrgapformula}
  Let $S,T\in \mathcal C(H_1,H_2)$ be densely defined. Then the operators $\widehat{T}^{\frac{1}{2}}S\check{S}^{\frac{1}{2}},\; T\check{T}^{\frac{1}{2}}\check{S}^{\frac{1}{2}},\; S\check{S}^{\frac{1}{2}}\check{T}^{\frac{1}{2}}$ and $\widehat{S}^{\frac{1}{2}}T\check{T}^{\frac{1}{2}}$ are bounded and
  \begin{equation}
  \theta(S,T)=\max{\{\left  \|   T\check{T}^{\frac{1}{2}}\check{S}^{\frac{1}{2}}-\widehat{T}^{\frac{1}{2}}S\check{S}^{\frac{1}{2}}\right \|, \|S\check{S}^{\frac{1}{2}}\check{T}^{\frac{1}{2}}-\widehat{S}^{\frac{1}{2}}T\check{T}^{\frac{1}{2}}\|}\}.
  \end{equation}
\end{theorem}

\section{Denseness of minimum attaining operators}
In this section we discuss the denseness of minimum attaining operators. First let us consider the case of functionals:

Let $H$ be a Hilbert space and $\phi: H\rightarrow \mathbb C$ be a non zero linear functional. Then $\phi$ is continuous (bounded) if and only if $\phi$ is closed. Since $H$ is infinite dimensional and $H/N(\phi)^{\bot}$ is isomorphic with $\mathbb C$, we can clearly conclude that $N(\phi)\neq {\{0}\}$. Hence $\phi$ is minimum attaining. Thus, the class of minimum attaining bounded linear functionals coincide with the space of all bounded linear functionals. So in this case, the minimum attaining functionals are dense.
If $H$ is finite dimensional, then clearly every linear functional is minimum attaining. Hence in this case also the result holds trivially.

In  the above discussion we can replace  Hilbert space by a Banach space. By a theorem of James we can conclude that a normed linear space $X$ is reflexive if and only if every non zero bounded linear functional is norm attaining (see \cite{james1,james2} for details ).  This is no more true if we replace the norm attaining property of functionals by minimum attaining property, as we have noted in the above paragraph.

Now we consider the case of densely defined closed operators defined between two different Hilbert spaces.
    We prove that the set of all minimum attaining densely defined closed operators defined between two Hilbert spaces is dense in the class of all densely defined closed operators with respect to the gap metric. First, we prove a key result related to the gap between two closed operators.

\begin{theorem}\label{specialgapformula} Let $S,T\in \mathcal C(H_1,H_2)$ be densely defined and $D(S)=D(T)$. Then
\begin{enumerate}
\item the operators
$\widehat T^\frac{1}{2}(T-S) \check S^\frac{1}{2}$ and $\widehat S^\frac{1}{2}(S-T) \check T^\frac{1}{2}$  are bounded and
\begin{equation*}
  \theta(S,T)=\max \; \Big\{\| \widehat T^\frac{1}{2}(T-S) \check S^\frac{1}{2}\|,\; \| \widehat S^\frac{1}{2}(S-T) \check T^\frac{1}{2}\| \Big\}
  \end{equation*}

\item if $T-S$ is bounded, then $\theta(S,T)\leq \|S-T\|$.
\end{enumerate}
\end{theorem}
\begin{proof}
First we simplify the term $T\check{T}^{\frac{1}{2}}\check{S}^{\frac{1}{2}}$. For $x\in H_1$, we have that  $\check{S}^{\frac{1}{2}}x\in D(S)=D(T)$, consequently,
$T\check{T}^{\frac{1}{2}}\check{S}^{\frac{1}{2}}x=\Hat{T}^{\frac{1}{2}}T\check{S}^{\frac{1}{2}}x$, by (\ref{contfnlcalculus-groetsch}) of Lemma \ref{boundedclosure}. Thus,
\begin{align}
 T\check{T}^{\frac{1}{2}}\check{S}^{\frac{1}{2}}-\widehat{T}^{\frac{1}{2}}S\check{S}^{\frac{1}{2}}x&=\left(\Hat{T}^{\frac{1}{2}}T\check{S}^{\frac{1}{2}}-\widehat{T}^{\frac{1}{2}}S\check{S}^{\frac{1}{2}}\right)x\\
                                                                                                                                                                         &=\widehat T^\frac{1}{2}(T-S) \check S^\frac{1}{2}x.
\end{align}
With a similar argument we can show that
\begin{equation*}
S\check{S}^{\frac{1}{2}}\check{T}^{\frac{1}{2}}-\widehat{S}^{\frac{1}{2}}T\check{T}^{\frac{1}{2}}= \widehat S^\frac{1}{2}(S-T)\check T^\frac{1}{2}.
\end{equation*}
Now the conclusion follows by Theorem \ref{shkgrgapformula}.

If $T-S$ is bounded, then
\begin{equation}
\| \widehat T^\frac{1}{2}(T-S) \check S^\frac{1}{2}\|\leq \| \widehat T^\frac{1}{2}\|\, \|T-S\|\, \|\check S^\frac{1}{2}\|\leq \|T-S\|.
\end{equation}

Similarly, we can conclude that $\|\widehat S^\frac{1}{2}(S-T)\check T^\frac{1}{2}\|\leq \|T-S\|$.

Hence  by the above two observations the conclusion follows.
\end{proof}

\begin{prop}\cite[Propositions 3.8, 3.9]{shkgr7}
Let $T\in \mathcal C(H)$ be positive. Then
\begin{enumerate}
\item $T\in \mathcal{M}_c(H)$ if and only if $T^{\frac{1}{2}}\in \mathcal{M}_c(H)$
\item  $T\in \mathcal{M}_c(H)$ if and only if $m(T)$ is an eigenvalue of $T$.
\end{enumerate}
\end{prop}
\begin{prop}\label{minnumradius}\cite[Proposition 3.5]{shkgr7}
Let $T\in \mathcal{C}(H)$ be positive. Then \begin{equation*}
m(T)=\inf{\{\langle Tx,x\rangle: x\in S_{D(T)}}\}.
\end{equation*}
\end{prop}

\begin{theorem}\label{densenesspositive}
Let $T\in \mathcal C(H)$ be positive. Then for each $\epsilon>0$, there exists an operator $S\in \mathcal B(H)$
 such that
 \begin{enumerate}
\item \label{epsilonnorm} $\|S\|\leq \epsilon$
\item \label{perturbednormattaining}$T+S$ is minimum attaining
\item \label{gapperturb} $\theta(S+T,T)\leq \epsilon$.
\end{enumerate}
\end{theorem}
\begin{proof}
We prove the results by considering the following cases.

Case ($1$): $m(T)>0$\\ First we may assume that $0<\epsilon<m(T)$.
 Since, $T\geq 0$ and $m(T)=\displaystyle \inf_{x\in S_{D(T)}}\langle Tx,x\rangle$ by Proposition \ref{minnumradius}, there exists  $x_{\epsilon}\in S_{D(T)}$, such that
\begin{equation}\label{norminequality}
\langle Tx_{\epsilon},x_{\epsilon}\rangle< m(T)+\frac{\epsilon}{2}.
\end{equation}
Now, define
\begin{equation}\label{rankoneoperator}
C_{\epsilon}(x)=\epsilon \,\langle x,x_{\epsilon}\rangle x_{\epsilon}.
\end{equation}
 Then clearly, $C_{\epsilon}$ is a rank one positive, bounded operator with $\|C_{\epsilon}\|=\|C_{\epsilon}(x_{\epsilon})\|=\epsilon$.

Let $T_{\epsilon}=T-C_{\epsilon}$.  Clearly,  $T_{\epsilon}$ is self-adjoint. In fact, we show that $T_{\epsilon}\geq 0$. To this end, let $x\in D(T_{\epsilon})=D(T)$. Then
\begin{align*}
\langle T_{\epsilon}x,x\rangle &=\langle Tx,x\rangle -\epsilon |\langle x,x_{\epsilon}\rangle|^2\\
                                               &\geq (m(T)-\epsilon )\langle x,x\rangle  \; \;(\text{by Cauchy-Schwarz inequality}).
\end{align*}
 In fact,
 \begin{equation}\label{lowernumboundeq}
 \|T_{\epsilon}^{\frac{1}{2}}x\|^2\geq (m(T)-\epsilon)\, \|x\|^2 \; \text{ for all} \; x\in D(T_{\epsilon}).
 \end{equation}
  Note that $D(T_{\epsilon})\subseteq D(T_{\epsilon}^{\frac{1}{2}})$. Using the fact that $D(T_{\epsilon})$ is a core for $D(T_{\epsilon}^{\frac{1}{2}})$, we can conclude that the inequality (\ref{lowernumboundeq}) holds for all $x\in D(T_{\epsilon}^{\frac{1}{2}})$. Hence  $T_{\epsilon}^{\frac{1}{2}}$ is invertible and consequently, $T_{\epsilon}$ is invertible. So $m(T_{\epsilon})>0$.

  We claim that $T_{\epsilon}\in \mathcal M_c(H)$.  We show that $m(T_{\epsilon})\in \sigma_d(T_{\epsilon})$.  Assume that $m(T_{\epsilon})\in \sigma_{ess}(T_{\epsilon})$. Then by the Weyl's theorem we have
   $\sigma_{ess}(T_{\epsilon})=\sigma_{ess}(T)$. Note as $m(T)\in \sigma(T)$ and $m(T)$ is the smallest spectral value,  we can conclude that $m(T)\leq m(T_{\epsilon})$.

   But we have
\begin{equation*}
m(T_{\epsilon})\leq \langle T_{\epsilon}x_{\epsilon},x_{\epsilon}\rangle =\langle Tx_{\epsilon},x_{\epsilon}\rangle -\epsilon <m(T) -\frac{\epsilon}{2}< m(T).
\end{equation*}
Thus our assumption that $m(T_{\epsilon})\in \sigma_{ess}(T_{\epsilon})$ is wrong. As $m(T_{\epsilon})\in \sigma(T_{\epsilon})$, it must  hold that $m(T_{\epsilon})\in \sigma_d(T_{\epsilon})$. Consequently, $T_{\epsilon}\in \mathcal M_c(H)$.

Note that $T_{\epsilon}\in \mathcal C(H)$ and $T_{\epsilon}-T=C_{\epsilon}|_{D(T)}$ is a bounded operator with domain $D(T)$. By Theorem \ref{specialgapformula}, it follows that
\begin{equation*}
\theta(T_{\epsilon}, T)\leq \|T_{\epsilon}-T\|=\|C_{\epsilon}|_{D(T)}\|\leq \|C_{\epsilon}\|=\epsilon.
\end{equation*}

Take $S=-C_{\epsilon}$. Then $S$ satisfies the stated conditions.

Case $(2)$ $T$ is not one-to-one\\
Clearly $T$ is minimum attaining. In this case  $S=0$ satisfy the required properties.

Case $(3)$ $T$ is one-to-one and $m(T)=0$\\
We can use case $(1)$ to get the desired operator $S$. Note that $T+\dfrac{\epsilon}{2} I$ is positive and $m(T+\dfrac{\epsilon}{2} I)=\dfrac{\epsilon}{2}$. Hence by Case ($1$) above, there exists a positive rank one operator $C$ with $\|C\|\leq \dfrac{\epsilon}{2}$ such that $T+\dfrac{\epsilon}{2} I-C$ is minimum attaining and $\theta(T, T+\dfrac{\epsilon}{2} I-C)\leq \|\dfrac{\epsilon}{2} I-C\|\leq \epsilon$. Let $S=\dfrac{\epsilon}{2}I-C$. Then $S$ satisfy all the required conditions.
\end{proof}

Now we prove the above result for the general case.

\begin{theorem}\label{perturbofminattaining}
  Let $T\in \mathcal C(H_1,H_2)$. Then for each $\epsilon>0$ there exists an operator $S\in \mathcal B(H_1,H_2)$ with $\|S\|\leq \epsilon$ such that $S+T$ is minimum attaining and $\theta(S+T,T)\leq \epsilon$. Moreover, if $m(T)>0$ then $S$ can be chosen to be rank one operator.
\end{theorem}

\begin{proof}
  Let $T=V|T|$ be the polar decomposition of $T$. Applying Theorem \ref{densenesspositive} to $|T|$, there exists $A\in \mathcal B(H_1)$ with $\|A\|\leq \epsilon$ and $|T|+A$ is minimum attaining.

  Define $S=VA$. Then $S\in \mathcal B(H_1,H_2)$ with $\|S\|\leq \epsilon$. Next, we claim that $T+S$ is minimum attaining. By construction in Theorem \ref{densenesspositive}, we have that $S=0$ if $T$ is not one-to-one, hence in this  case clearly $T+S=T$ is minimum attaining.

  If $T$ is one-to-one, then $V$ is an isometry and $T+S$ is minimum attaining as $|T|+A$ minimum attaining. Note that $m(T+S)=m(|T|+A)$.

  In case if $m(T)>0$, then $A$ is a rank one operator and so is the operator $S$.

Finally, by Theorem \ref{specialgapformula}, we have $\theta(S+T,T)\leq \|S\|\leq \epsilon$.
\end{proof}

The following Corollary is an  immediate consequence of Theorem \ref{perturbofminattaining}.
\begin{cor}\label{densenessingaptop}
The set $\mathcal M_{c}(H_1,H_2)$ is dense in $\mathcal C(H_1,H_2)$ with respect to the gap topology.

\end{cor}
\begin{cor}
The set of all minimum attaining bounded operators is dense in $\mathcal B(H_1,H_2)$ with respect to the norm topology of $\mathcal B(H_1,H_2)$.
\end{cor}
\begin{proof}
The gap topology restricted to $\mathcal B(H_1,H_2)$ coincide with the norm topology of $\mathcal B(H_1,H_2)$ by \cite[Theorem 2.5]{nakamotogapformulas}.  Hence the conclusion follows by Corollary \ref{densenessingaptop}.
\end{proof}

\begin{theorem}\label{bddbelowdense}
Let   $$\mathcal  G:={\{T\in \mathcal M_c(H_1,H_2): T  \; \text{is densely defined and bounded below }}\}$$  and  $ C_{b}(H_1,H_2)={\{T\in \mathcal C(H_1,H_2): T\; \; \text{is bounded below}}\}$. If $T\in C_{b}(H_1,H_2)$ and  $\epsilon>0$, there exists $\tilde{T}\in \mathcal M_c(H_1,H_2)$ such that $\theta(T,\tilde{T})\leq \epsilon$.
\end{theorem}
\begin{proof}
Let $T\in \mathcal G$ and $T=V|T|$ be the polar decomposition of $T$. By the assumption, $m(|T|)>0$. Following the steps in Theorem \ref{densenesspositive}, for each $\epsilon>0$, we get a rank one operator $S$ with $\|S\|\leq \epsilon$ such that $|T|+S\in \mathcal M_c(H_1)$. Since $V$ is an isometry, we can conclude that $\tilde{T}:=V(|T|+S)=T+VS$ is minimum attaining and $m(\tilde{T})=m(|T|+S)$.

Since $\tilde{T}-T=VS$ is a bounded operator with domain $D(T)$, by Theorem  \ref{specialgapformula}, we can conclude that $\theta(\tilde{T},T)\leq \epsilon$.
\end{proof}

\begin{cor}
The set   $$\mathcal M_{b}(H_1,H_2):={\{A\in \mathcal B(H_1,H_2): A\; \text{is minimum attaining and  bounded below}}\}$$  is norm  dense in $\mathcal  B_{b}(H_1,H_2)={\{A\in \mathcal B(H_1,H_2): A\; \text{is bounded below}}\}$.
\end{cor}
\begin{proof}
The gap metric and the metric induced by the operator norm are equivalent  on $\mathcal B(H_1,H_2)$ by \cite[Theorem 2.5]{nakamotogapformulas}. Hence the conclusion follows by Theorem \ref{bddbelowdense}.
\end{proof}

\bibliographystyle{amsplain}
\bibliography{minbib}

\end{document}